\documentclass[11pt,letterpaper]{article}

\usepackage[top=2cm, bottom=4.5cm, left=2.5cm, right=2.5cm]{geometry}

\usepackage{amsmath, amsfonts, amsthm, verbatim,
	amssymb,dsfont,color, accents}

\usepackage[utf8x]{inputenc}
\usepackage{mathtools}

\usepackage{tikz}
\usetikzlibrary{arrows.meta,arrows,decorations.pathreplacing}
\RequirePackage[colorlinks,citecolor=blue,urlcolor=blue]{hyperref}

\newcommand{\R}{\mathbb{R}}

\newcommand{\N}{\mathbb{N}}
\newcommand{\E}{\mathbb{E}}

\renewcommand{\P}{\mathbb{P}}

\newcommand{\calf}{{\cal F}}

\theoremstyle{plain}
\newtheorem{theorem}{Theorem}[section]
\newtheorem{definition}[theorem]{Definition}
\newtheorem{lemma}[theorem]{Lemma}
\newtheorem{prop}[theorem]{Proposition}
\newtheorem{corollary}[theorem]{Corollary}
\newtheorem{condition}{Condition}
\newtheorem{Assumption}{Assumption}
\renewcommand{\theAssumption}{\Alph{Assumption}}
\makeatletter
\newcommand{\settheoremtag}[1]{
	\let\oldtheAssumption\theAssumption
	\renewcommand{\theAssumption}{#1}
	\g@addto@macro\endAssumption{
		\addtocounter{Assumption}{-1}
		\global\let\theAssumption\oldtheAssumption}
}
\makeatother
\theoremstyle{remark}
\newtheorem{remark}[theorem]{Remark}

\newcommand{\eps}{{\epsilon}}

\newcommand{\Ga}{{\Gamma}}

\newcommand{\si}{{\sigma}}

\newcommand{\Om}{{\Omega}}

\newcommand{\ov}{\overline}

\newcommand{\wh}{\widehat}
\newcommand{\wt}{\widetilde}

\newcommand{\lec}{\lesssim}

\newcommand{\Den}{\Delta_n}

\newcommand{\bthm}{\begin{theorem}}
	\newcommand{\ethm}{\end{theorem}}

\newcommand{\bcor}{\begin{corollary}}
	\newcommand{\ecor}{\end{corollary}}

\newcommand{\blem}{\begin{lemma}}
	\newcommand{\elem}{\end{lemma}}

\newcommand{\bprop}{\begin{prop}}
	\newcommand{\eprop}{\end{prop}}

\newcommand{\bcond}{\begin{condition}}
	\newcommand{\econd}{\end{condition}}

\newcommand{\bdf}{\begin{definition}}
	\newcommand{\edf}{\end{definition}}

\newcommand{\bex}{\begin{example}}
	\newcommand{\eex}{\end{example}}

\newcommand{\brem}{\begin{remark}}
	\newcommand{\erem}{\end{remark}}

\newcommand{\bpr}{\begin{proof}}
	\newcommand{\epr}{\end{proof}}

\newcommand{\benu}{\begin{enumerate}}
	\newcommand{\eenu}{\end{enumerate}}

\newcommand{\beq}{\begin{equation}}
	\newcommand{\eeq}{\end{equation}}

\newcommand{\bit}{\begin{itemize}}
	\newcommand{\eit}{\end{itemize}}

\newcommand{\bass}{\begin{Assumption}}
	\newcommand{\eass}{\end{Assumption}}

\newcommand{\bff}{\textbf}

\numberwithin{equation}{section}

\allowdisplaybreaks[3]

\newcommand{\lpnorm}[2]{\left|\left|{#1}\right|\right|_{L^{#2}}}

\DeclareMathOperator{\Cov}{Cov}
\DeclareMathOperator{\Var}{Var}

\allowdisplaybreaks

\graphicspath{ {./figures/} }

\title{Pre-averaging  fractional processes contaminated by noise,\\ with an application to turbulence}
\author{David Chen\thanks{Columbia College, Columbia University, 1130 Amsterdam Avenue, New York, NY 10027, USA, e-mails: \url{dc3451@columbia.edu}, \url{bjm2190@columbia.edu}}, Yu Cheng\thanks{Department of Earth and Planetary Sciences, Harvard University, 20 Oxford Street, Cambridge, MA 02138, USA, e-mail: \url{yucheng1@fas.harvard.edu}}, Carsten Chong\thanks{Department of Statistics, Columbia University, 1255 Amsterdam Avenue, New York, NY 10027, USA, e-mail: \url{chc2169@columbia.edu}}, Pierre Gentine\thanks{Department of Earth and Environmental Engineering and Earth Institute, Columbia University, 500 West 120th Street, New York, NY 10027, USA, e-mail: \url{pg2328@columbia.edu}},\\ Wangdong Jia\thanks{The Fu Foundation School of Applied Engineering and Science, Columbia University, 500 West 120th Street, New York, NY 10027, USA, e-mail: \url{wj2315@columbia.edu}}, Bryce Monier$^\ast$\!\!\!, and Shiyang Shen\thanks{School of General Studies, Columbia University, 2970 Broadway, New York, NY 10027, USA, e-mail: \url{ss6101@columbia.edu}}}
\date{}

\begin{document}
\maketitle

\begin{abstract}
In this article, we consider the problem of estimating fractional processes based on noisy high-frequency data. Generalizing the idea of   pre-averaging to a fractional setting, we exhibit a sequence of consistent estimators for the unknown  parameters of interest by proving a law of large numbers for associated variation functionals. In contrast to the semimartingale setting, the  optimal window size for pre-averaging depends on the unknown roughness parameter of the underlying process. We evaluate the performance of our estimators in a simulation study and use them to   empirically verify Kolmogorov's $2/3$-law in  turbulence data contaminated by instrument noise.
\end{abstract}

\bigskip



\section{Introduction}

In 1941, Kolmogorov and Obukhov   derived some universal statistical properties of turbulent fields based on dimensional arguments (see \cite{K41} or \cite[Chapter 6]{Frisch95}). For example, \emph{Kolmogorov's $2/3$-law} postulates that the velocity structure function in a stationary homogeneous isotropic turbulent flow at sufficiently high Reynolds numbers satisfies the scaling property
\beq\label{eq:23} \E [ ((v(t,x+r)-v(t,x))\cdot \tfrac r{\lvert r\rvert} )^2 ] \sim C\lvert r\rvert^{2/3}, \eeq
where $v(t,x)$ is the velocity at time $t$ and position $x$ and the radius $r $ belongs to a certain \emph{inertial range} of scales. The corresponding relationship in the spectral domain is known as \emph{Kolmogorov's $-5/3$-law}.

The $2/3$- and $-5/3$-laws have been subjected to extensive experimental validation. In many cases, instead of verifying the spatial property \eqref{eq:23} directly, researchers analyze temporal measurements of $v(t,x)$ at a fixed spatial point $x$ (because such data are typically easier to obtain). Substituting temporal for spatial data is justified under the assumption of \emph{Taylor's frozen turbulence hypothesis} \cite{T38} according to which $v$ satisfies the same scaling in time, that is,
\beq\label{eq:23-2}  \E [ ((v(t,x+U\tau)-v(t,x))\cdot \tfrac r{\lvert r\rvert} )^2 ] \approx \E [ ((v(t+\tau,x)-v(t,x))\cdot \tfrac r{\lvert r\rvert} )^2 ] \sim C  \tau ^{2/3}, \eeq
where $U$ is the mean flow velocity in the $x$-direction.
The same scaling properties are expected to hold for scalar quantities such as temperature \cite{CLASG20}.

While the $2/3$- and $-5/3$-laws are largely supported by field experiments, deviations from the hypothesized scaling exponents have been observed as a consequence of, for example, stratification \cite{CLASG20}, the failure of Taylor's hypothesis \cite{Cheng17,Roy21} or measurement errors \cite{CLASG20, Dilling17,Durgesh14}. In this paper, we are interested in the question how to distinguish whether  deviations from the $2/3$- and $-5/3$-laws are due to physically relevant reasons (e.g., stratification, failure of Taylor's hypothesis) or simply measurement noise. 

Let us now frame this question in a more general mathematical framework: we  consider a broad class of stochastic processes of the form
 \begin{equation}\label{eq:X}
    X_t = X_0 + A_t + X_t^H = X_0 + \int_0^ta_sds + K_H^{-1}\int_0^t(t-s)^{H - \frac{1}{2}}\sigma_sdB_s,
  \end{equation}
where $B$ is a Brownian motion and  $H\in(0,1)$, called the \emph{roughness parameter} of $X$, is a measure of regularity for the process $X$. We included a  stochastic drift $a$ and a stochastic volatility $\si$ in \eqref{eq:X} to account for possible non-stationarity in the model. The number $$K_H=\frac{\Ga(H+\frac12)}{\sqrt{\Ga(2H+1)\sin(\pi H)}}$$ is   a normalizing constant. Under mild assumptions, the second moments of $X$ satisfy
\[ \E[(X_{t+\tau}-X_t)^2] \sim C\tau^{2H}  \]
for small $\tau$, while its
 spectral density $S$  satisfies
$$ S(f) \sim C'f^{-2H-1}$$
for large $f$ and some  $C,C'>0$ (see \cite{CHPP13}). In particular,  Kolmogorov's 2/3- and $-5/3$-laws correspond to $H=\frac13$ in this model.

In the absence of measurement noise, \cite{BNCP11,CHPP13} derived consistent and asymptotically mixed normal estimators of $H$ from high-frequency observations of $X$ and subsequently applied them to turbulence data from \cite{Dhruva00}; see also \cite{Bennedsen20}. In this paper, we are interesting in how to recover $H$ from noisy observations
\begin{equation} \label{eq:Z} Y_t = X_t + \rho_t Z_t, \end{equation}
where $Z_t$ is a white noise sequence and $\rho$ is a noise volatility process. 

In the case where $H=\frac12$, that is, if $X$ is a continuous It\^o semimartingale, this statistical problem has been intensively studied in financial econometrics (see \cite{BNHLS08,JLMPV09,PV09,ZMA05} and also the references in \cite[Chapter 7]{AJ14}). Much less is known if $H\neq \frac12$ and observations are noisy. In \cite{GH07}, the authors derived rate-optimal estimators of $H$ in the case where $a$ and $\si$ in \eqref{eq:X}  are constant and $H>\frac12$; in \cite{Szymanski22}, rate-optimal estimators and associated central limit theorems were obtained for all values of $H$, but still with constant $a$ and $\si$ (and also Gaussian $Z_t$). 

In this paper, we pursue a different direction by proposing a non-parametric estimator that is still consistent for $H$ if $a$, $\si$ and $\rho$ are random and time-varying and $Z$ is not necessarily Gaussian. This will be achieved by generalizing the pre-averaging approach of \cite{JLMPV09,PV09} to a fractional setting. The remaining paper is organized as follows: In Section~\ref{sec:main}, we introduce the estimators we propose and state the main technical result  of this paper (Theorem~\ref{thm:main}). In the same section, we will then use Theorem~\ref{thm:main} to construct consistent estimators of $H$ in Theorem~\ref{thm:H}. All proofs are postponed to Section~\ref{sec:proof}. In Section~\ref{sec:MC}, we carry out a Monte--Carlo simulation before we apply our estimators to noisy turbulence data in Section~\ref{sec:data}.

\section{Main result}\label{sec:main}

Consider a filtered probability space $(\Om,\calf',(\calf'_t)_{t\geq0},\P)$ together with a smaller $\si$-field $\calf$ and a smaller filtration $(\calf_t)_{t\geq0}$ satisfying $\calf_t\subseteq \calf'_t\cap \calf$ for all $t\geq0$. Our assumptions on the ingredients of \eqref{eq:X} and \eqref{eq:Z} are as follows.

\begin{Assumption}\label{ass:A}
Suppose that $X$ and $Y$ are given by \eqref{eq:X} and \eqref{eq:Z}, respectively, 
where $B$ is a standard $(\calf_t)_{t\geq0}$-Brownian motion, $X_0$ is an $\calf_0$-measurable random variable and $a$, $\si$ and $\rho$ are locally bounded $(\calf_t)_{t\geq0}$-adapted stochastic processes that are furthermore continuous in probability. Moreover, we assume that $(Z_t)_{t\geq0}$ is centered with unit variance and uniformly (in $t$) bounded moments of all orders, adapted to $(\calf'_t)_{t\geq0}$ and independent of $\calf$ and for different values of $t$.
\end{Assumption}

Our goal is to consistently estimate $H$ based on equally spaced high-frequency noisy observations
\[ \{ Y_{\frac in}: i=0,\dots, \lfloor nT\rfloor\} \]
of $X$,
where $T>0$ is a finite time horizon and $n$ is the number of observations per unit time. From our estimation procedure, we also obtain estimators of $C_t=\int_0^T \si_t^2 dt$ and $\Pi_t = \int_0^T \rho_t^2 dt$, the integrated volatility and integrated noise volatility.
In order to mitigate the impact of the noise $Z_t$, we rely on the pre-averaging approach of \cite{JLMPV09,PV09}. To this end, we choose a weight function \(g: [0,1] \rightarrow \R\) and a window size $k_n = \frac{n^\kappa}{\theta}$, where $\theta,\kappa>0$ are tuning parameters. Given a function $W$, random or not (but we have  $W=Y$ in mind), we write $g^n_j = g ( \frac{j}{k_n}  )$ and define
\begin{equation}\label{eq:prev-W}\begin{split}
 \overline{W}(g)^n_i = \sum_{j=1}^{k_n-1}g^n_j \Big( W_{\frac{i+j-1}{n}} - W_{\frac{i+j-2}{n}}  \Big) = \sum_{j=1}^{k_n-1}g^n_j\Delta_{i+j-1}^n W, \\
    \widehat{W}(g)^n_i = \sum_{j=1}^{k_n} (g^n_j - g^n_{j-1} )^2 (\Delta_{i+j-1}^n W )^2 = \sum_{j=1}^{k_n} (\Delta g^n_j )^2 (\Delta_{i+j-1}^n W )^2.
\end{split}\end{equation}
  Morally, the former  expression  is a \textit{weighted average} of the increments of $W$ (hence the name ``pre-averaging''), which we consider for $W=Y=X+Z$ in order to  smooth out  the noise component $Z$ while preserving the signal component $X$. Our main functionals of interest given by
 \begin{equation}\label{eq:Vnf}
    V(g)^{n,f}_T(Y) = \frac{1}{n}\sum_{i=1}^{ \lfloor nT  \rfloor}f\bigg( \frac{\overline{Y}(g)^n_i}{ ( k_n/n  )^H}, \frac{\widehat{Y}(g)^n_i}{ ( k_n/n  )^{2H}} \bigg)
  \end{equation} 
for some test function $f:\R^2 \to \R$.
%
The following theorem is our main technical result:
\begin{theorem}\label{thm:main}
In the setting of Assumption~\ref{ass:A}, further suppose that \(g: [0,1] \rightarrow \R\) is \(C^2\) with $g(0)=g(1)=0$ and 
 \(f: \R^2 \rightarrow \R\) is \(C^2\) with all partial derivatives up to order 2 of at most polynomial growth. Then, for any $T>0$, $\theta>0$ and  \(\kappa \in  [\frac{2H}{2H+1}, 1 )\), we have the convergence
  \begin{equation}\label{eq:LLN}
    V(g)^{n,f}_T(Y)  \overset{\mathbb{P}}{\longrightarrow} V(g)^{f}_T(Y)=\int_0^T \mu_f\bigg( \sigma_s^2\eta ( g  ), \Theta^{2H+1}\rho^2 \int_0^1g'(r)^2dr \bigg)ds,
  \end{equation}
  where
   \begin{equation}\label{eq:muf}
  	\mu_f(v_1, v_2) = \E\big[ f(\sqrt{v_1}Z_1 + \sqrt{v_2}Z_2, 2v_2) \big]
  \end{equation}
  with independent \(Z_1, Z_2 \sim N(0,1)\) 
  and
  \begin{equation}
    \Theta =
    \begin{cases}
      0 & \kappa \neq \frac{2H}{2H+1}, \\
      \theta &  \kappa = \frac{2H}{2H+1},
    \end{cases}\qquad
        \eta(g) = \frac12 \int_0^1\int_0^1g'(y)\lvert y-x\rvert^{2H} g'(x) dxdy.
\label{eq:etag}
  \end{equation}
\end{theorem}

Next, we explain how Theorem~\ref{thm:main} can be used to solve the original statistical problem. Note that $V(g)^{n,f}_T(Y)$ is not a statistic, as it involves the unknown parameter $H$. In order to circumvent this problem, we use a classical \emph{change-of-frequency} approach, see \cite{BNCP11,CHPP13}. More precisely, we consider 
$f(x,y)=x^2$ and  
form the ratio statistic
\begin{equation}\label{eq:R} 
	R(g)^{n,2}_T(Y)= \frac{\sum_{i=1}^{\lfloor nT/2 \rfloor}(\overline{Y}(g)^{n/2}_{i})^{2}}{\sum_{i=1}^{\lfloor nT \rfloor}(\overline{Y}(g)^{n}_{i})^{2}}, 
\end{equation} 
where the numerator is computed with half the original frequency (i.e., by taking increments of step size $\frac 2n$.)

%
%

\begin{theorem}\label{thm:H} Grant the assumptions of Theorem~\ref{thm:main}. Then the estimator
	\begin{equation}\label{eq:Hn}
		\widehat{H}_n = \frac{1}{2(1 - \kappa)}\big(1+\log_{2} R(g)^{n,2}_T(Y) \big)
	\end{equation}
is consistent for $H$, that is, satisfies
	 $\widehat{H}_n \to H$ in probability as $n \to \infty$.
\end{theorem}

\begin{proof}
Let $f(x,y)=x^2$. Then by Theorem~\ref{thm:main}, both $V(g)^{n/2,f}_T(Y)$ and $V(g)^{n,f}_T(Y)$ converge in probability to the same limit. In conjunction with  the fact that $k_n\sim n^\kappa/\theta$, we obtain
	\begin{equation*}
	R(g)^{n,2}_T(Y)	=\frac{\frac n2 (2k_{n/2}/n)^{2H} V(g)^{n/2,f}_T(Y)}{n (k_n/n)^{2H}V(g)^{n,f}_T(Y)} \stackrel{\P}{\longrightarrow} 2^{2H-1-2H\kappa}
	\end{equation*}
from which we easily deduce $\widehat{H}_n \to H$ in probability after applying the transform on the right-hand side of \eqref{eq:Hn}.
%
\end{proof}

Next, we exhibit consistent estimators of $\Pi_t=\int_0^T \rho_t^2 dt$ and $C_T=\int_0^T \si_t^2dt$, the respective integrated variance of noise and signal.

\begin{prop}
	We assume the hypotheses of Theorem~\ref{thm:main}. Furthermore, let $F(x,y)=x^2-y/2$ and define
	\begin{equation}\label{eq:Pi}
		\widehat{\Pi}^n_T = \frac{1}{2\lfloor nT \rfloor}\sum_{i=1}^{\lfloor nT \rfloor} \Big(Y_{\frac{i}{n}} - Y_{\frac{i-1}{n}}\Big)^{2}
	\end{equation}
and
\begin{equation}\label{eq:C}
\wh C^n_T=	\frac{1}{\widehat{\eta}(g)}\wh V(g)_{T}^{n,F}(Y), 
\end{equation}
where $\wh V(g)_{T}^{n,F}(Y)$ and $\widehat{\eta}(g)$ are computed from \eqref{eq:Vnf} and \eqref{eq:etag}, respectively, but with the estimator
$\widehat{H}_n$ instead of $H$.
	Then, as $n\to\infty$, we have $\Pi^n_T\to \Pi_T$ and $\wh C^n_T \to C_T$ in probability.
\end{prop}

\begin{proof} The convergence in \eqref{eq:C} is a direct consequence of Theorem~\ref{thm:main}, Theorem~\ref{thm:H} and the fact that $\mu_F(v_1,v_2) = v_1+v_2 - (2v_2)/2=v_1$. For \eqref{eq:Pi}, 
	since $Y_{t} = X_{t} + \rho_t Z_{t}$, we may write the difference as
	$\Delta^n_i Y = \Delta^n_i X + \Delta^n_i \rho Z$. Squaring, we see that
	\begin{equation} \label{rho-estim}
		\wh\Pi^n_T = \frac{1}{2\lfloor nT\rfloor}\sum_{i=1}^{\lfloor nT \rfloor} \big((\Delta^n_i X)^2 +\Delta^n_i X \Delta^n_i \rho Z+(\Delta^n_i \rho Z)^2\big).
	\end{equation}
Since $\E[\lvert\Delta^n_i X\rvert^p]^{1/p}\lec\Den^H$ (see Lemma~\ref{lem:fBM-bound} below), the first two terms are negligible and we only need to determine the limit of 
\[\frac{1}{2\lfloor nT\rfloor}\sum_{i=1}^{\lfloor nT \rfloor} (\Delta^n_i \rho Z)^2.\]
Conditionally on $\calf$, $(\Delta^n_i \rho Z)_{i\in\N}$ is a $1$-dependent sequence (i.e., $\Delta^n_i \rho Z$ and $\Delta^n_j \rho Z$ are $\calf$-conditionally independent for any $\lvert i-j\rvert>1$). Thus, we obtain \eqref{eq:Pi} by an application of the weak law of large numbers.
\end{proof}

\begin{remark}\label{rem:opt}
Our estimator of $H$ depends on a choice of $\kappa\in [\frac{2H}{2H+1},1]$. Although we cannot prove it in this article, we strongly conjecture that choosing $\kappa_\textrm{opt}=\frac{2H}{2H+1}$ yields a rate-optimal estimator of $H$. But taking $k_n= n^{\kappa_\textrm{opt}}/\theta$ directly is impossible because we do not know $H$. To solve this problem and choose a sequence that is asymptotically of order $n^{\kappa_\textrm{opt}}$, we could, in principle, follow the ideas outlined in \cite{CHLRS22a, CHLRS22b, Szymanski22}. But since optimality is not the main focus of this work, we do not present the details here. Instead, in the simulation and empirical sections below, we follow a more \emph{ad hoc} approach, which yields sufficiently good estimates in our application.
\end{remark}

\section{Monte--Carlo simulations}\label{sec:MC}

In this section, we document the finite-sample performance of our estimators from Section~\ref{sec:main} based on 5{,}000 simulated paths from the model
\begin{equation}\label{eq:model} 
	Y_t = \si B^H_t + \rho Z_t, \quad t\in[0,1],
\end{equation} 
for each of the values
\[ H\in \{ 0.1, 0.2, 0.3, \tfrac13, 0.4, 0.5, 0.6, 0.7, 0.8, 0.9\}. \]
In \eqref{eq:model}, $B^H$ is a standard fractional Brownian motion with Hurst index $H$, $Z_t$ is  standard Gaussian white noise and $\si=1$ and $\rho=0.1$. As a sampling frequency, we choose $n=10{,}000$. Furthermore, we take  $g(x) = 2(x \wedge (1 - x))$ 
as  the weight function for pre-averaging. In order to obtain a window size $k_n$ that is as close as possible to the optimal rate $n^{2H/(2H+1)}$, we use an adaptive approach by first taking $k_n^{(0)}=n^{2/3}$, which corresponds to an admissible 
	$\kappa$ for any value of $H\in(0,1)$. We then obtain an estimator $\wh H_n^{(0)}$ of $H$ using the formula \eqref{eq:Hn} and plug in this value to get a new window size $k_n^{(1)}=n^{2 \wh H_n^{(0)}/(2\wh H_n^{(0)}+1)}$. We then update our estimator of $H$ to $\wh H_n^{(1)}$ and subsequently compute a new window size $k_n^{(2)}$. This procedure is repeated until $H$ no longer changes by more than a fixed threshold of $0.025$ or 100 iterations have been computed, whichever happens first. (If an estimate of $H$ is negative at some point in this procedure, we continue the iteration with 
	$\kappa=0$.) With the final estimate of $H$, we then compute estimates of $\si^2$ and $\rho^2$ via \eqref{eq:Pi} and \eqref{eq:C}. Table~\ref{tab:1} lists the bias, standard error (SE) and root mean square error (RMSE) of the resulting estimators. As we can see, they work very  well unless $H$ is close to $0$ or $1$. In particular, their precision is very satisfactory for $H=\frac13$, which is the relevant value for applications in turbulence.
	
	\begin{table}[htbp]
		\begin{center}
			\begin{tabular}{c | c c c | c c c | c c c}
			 $H$ & ~&$\wh H_n$&~&~&$\wh C^n_T$&~&~&$\wh \Pi^n_T$&~\\
			 ~&Bias & SE & RMSE &Bias & SE & RMSE &Bias & SE & RMSE  \\
			 \hline
		0.1&	 0.163&	0.006&	0.163&	473.68\%&	2.61\%	&473.69\%	&-56.67\%&	0.85\%&	56.68\%
\\
		0.2&	 -0.004	&0.011&	0.011&	44.89\%	&0.67\%	&44.89\%&	-3.42\%	&2.76\%	&4.39\%
\\
		0.3&	 0.006&	0.018&	0.019&	4.25\%&	0.50\%&	4.28\%&	-1.31\%	&2.71\%	&3.01\%
\\
		\bff{1/3}&	\bff{0.001}&	\bff{0.019}	&\bff{0.019}&	\bff{1.94\%}&	\bff{0.48\%}&	\bff{2.00\%}&	\bff{-0.40\%}	&\bff{3.06\%}&	\bff{3.09\%}
\\
		0.4&	 -0.002&	0.025&	0.025&	0.41\%&	0.48\%	&0.63\%	&-0.03\%	&3.98\%	&3.98\%
\\
			0.5& -0.002	&0.033&	0.033&	0.04\%&	0.47\%	&0.47\%&	0.10\%	&6.11\%	&6.11\%
\\
0.6&			 -0.006&	0.043&	0.043&	0.01\%&	0.48\%&	0.48\%&	-0.63\%&	9.15\%&	9.17\%
\\
	0.7&		 -0.008&	0.052&	0.053&	0.00\%&	0.47\%&	0.47\%&	-0.52\%	&14.27\%&	14.28\%\\
		0.8&	 -0.017	&0.064&	0.066&	0.00\%&	0.48\%&	0.48\%&	-1.28\%	&27.10\%	&27.13\%
\\
			0.9& -0.028	&0.076	&0.081	&0.01\%&	0.47\%	&0.47\%&	-2.89\%	&58.99\%&	59.06\%\\
			 \hline
%
			
			\end{tabular}
					\caption{Monte--Carlo results based on $5{,}000$   replications. For $\wh C^n_T$ and $\wh \Pi^n_t$, bias, SE and RMSE are stated in percentage of the true values. The row $H=\frac13$, highlighted in bold font, is the case of interest in the turbulence application.} 
			\label{tab:1}
		\end{center}
	\end{table}

\section{Application to turbulence data}\label{sec:data}

To exemplify the performance of our estimators, we apply them to a dataset of 10 Hz temperature measurements collected at the Concordia Station (Dome C, Antarctica) on January 9, 2015, and obtained  from  the IPEV/PNRA Project ``Routin Meteorological Observation at Station Concordia'' (see \url{http://www.climantartide.it}). We also refer to \cite{CLASG20} for a more detailed description of the data. Under the hypothesis that temperature measurements satisfy Kolmogorov's 2/3-law (or, in the spectral domain, Kolmogorov's $-5/3$-law), we should see a roughness parameter $H=\frac13$ for a whole range of frequencies in the so-called inertial range. Thus, for each frequency
\[ \Delta_n \in \{ 1\,\mathrm{s}, \dots, 30\,\mathrm{s} \} \]
 we split our dataset into $\Delta_n/0.1\,\mathrm{s}$ separate time series with increments of length $\Delta_n$ and compute, for each hour of January 9, 2015, estimates of $H$, volatility and noise volatility by averaging over the estimates obtained for each of the $\Delta_n/0.1\,\mathrm{s}$  time series. Figure~\ref{fig:2} shows the resulting estimates of $H$, $C_T$ and $\Pi_T$, obtained by averaging (for $H$) and summing up (for $C_T$ and $\Pi_T$) the hourly estimates, as a function of $\Delta_n$. In particular, we find that after removing the noise, the estimated roughness parameter of the data is in very good agreement with the theoretically predicted value of $H=\frac13$. On the other hand, if we do not perform pre-averaging (i.e., we choose $\kappa=0$), the estimates of $H$ are practically $0$ for all frequencies, indicating the presence of measurement noise in the data.

 \begin{figure}[h!]
 	\centering
 	\includegraphics[width=0.49\linewidth]{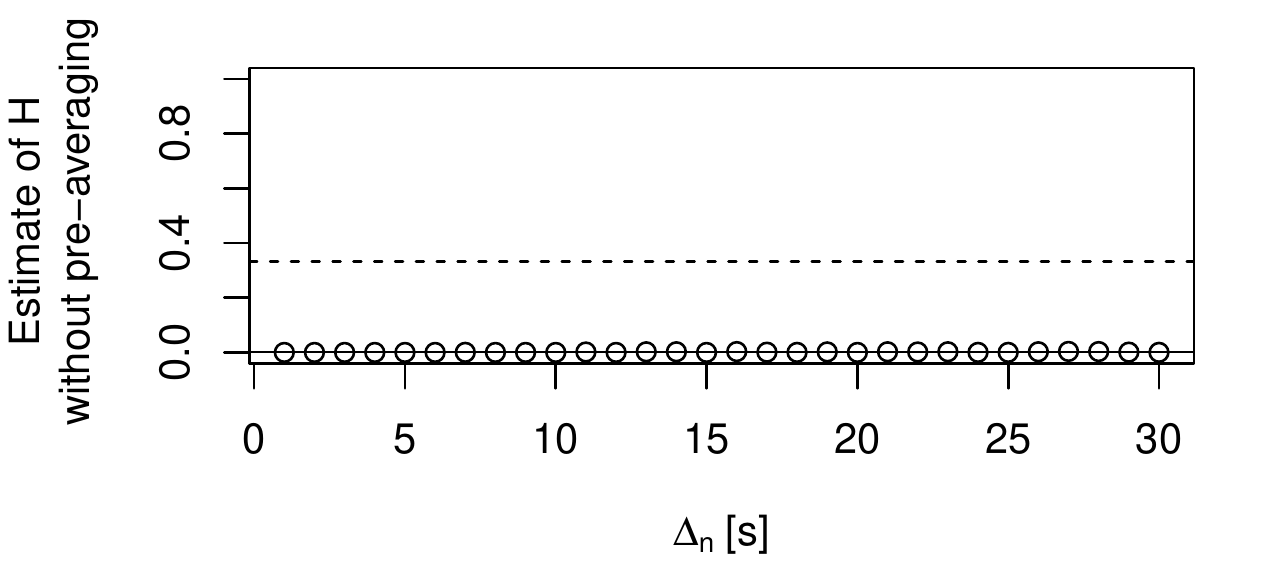}
 	 	\includegraphics[width=0.49\linewidth]{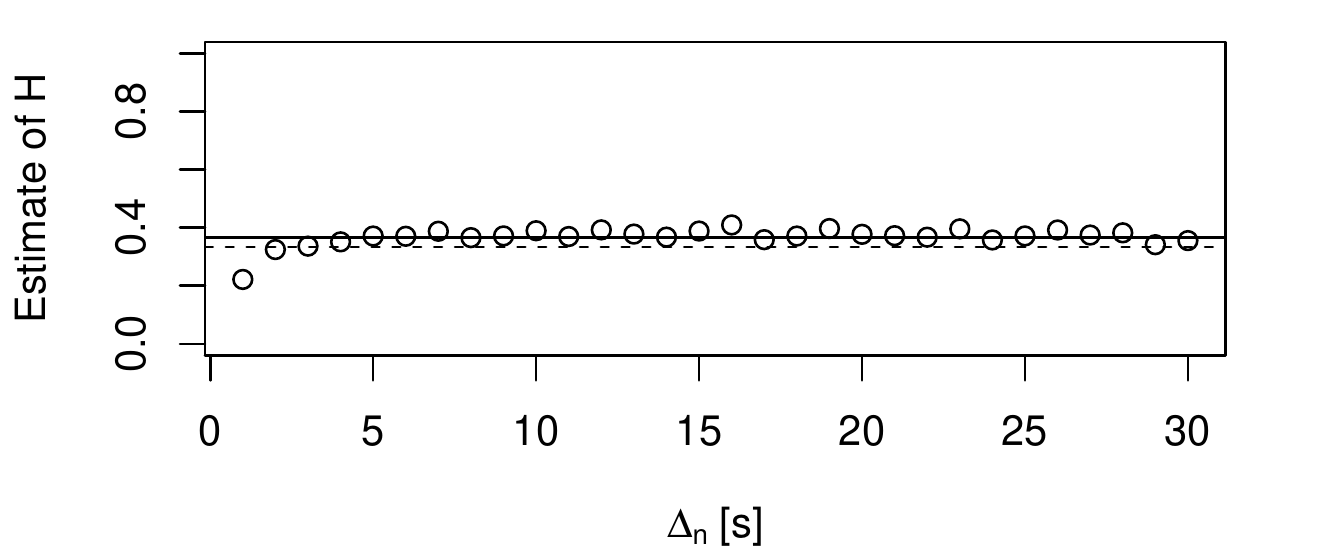}
 	 	\includegraphics[width=0.49\linewidth]{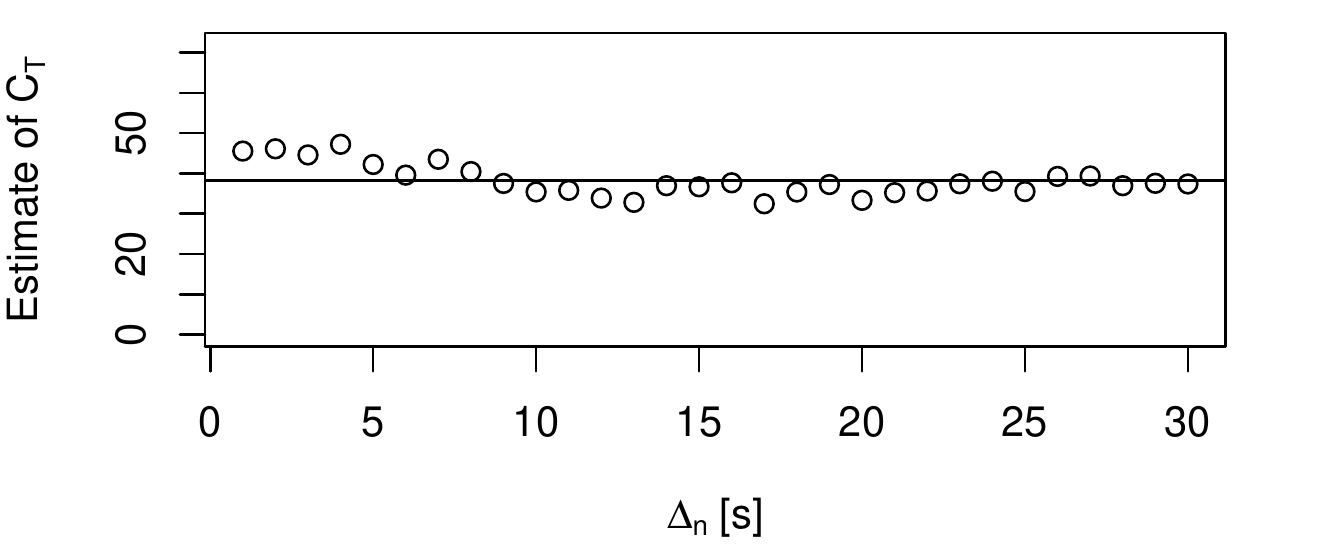}
 	 	\includegraphics[width=0.49\linewidth]{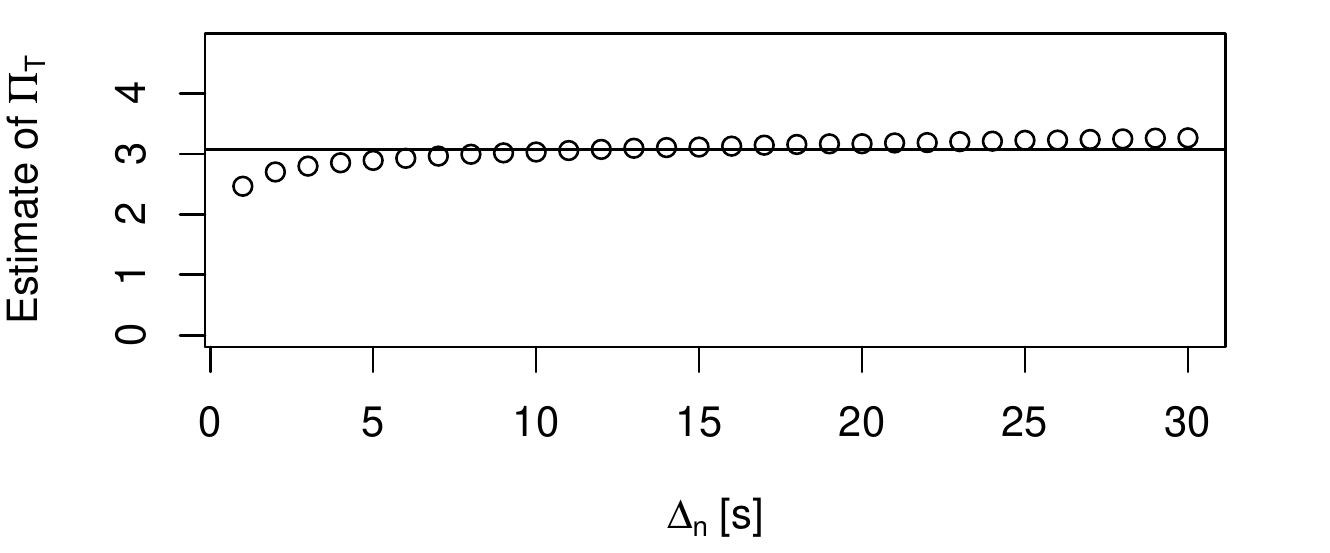}
 	 	\caption{Estimates of $H$ without pre-averaging and estimates of $H$, $C_T$ and $\Pi_T$ with pre-averaging for different sampling frequencies $\Delta_n$. The solid lines indicate the mean over all frequencies. In the first and second plot, the dashed line corresponds to the hypothesized value of $H=\frac13$.}
 	 	\label{fig:2}
 \end{figure}

\section{Proof of Theorem~\ref{thm:main}}\label{sec:proof}

Throughout this section, all conditions in Theorem~\ref{thm:main} are tacitly assumed. Also, by a standard localization argument (see e.g., \cite[Lemma 3.4.5]{JP12}), we can additionally make the following assumptions:
\begin{itemize}
\item The processes $\si$, $a$ and $\rho$ are uniformly bounded by a fixed constant. In particular, they possess uniformly bounded moments of all orders.
\item The processes $\si$, $a$ and $\rho$  are $L^p$-continuous for any $p>0$.
\end{itemize}
These two conditions will also be used without mentioning in what follows. We further write $A\lec B$ if there is $C\in(0,\infty)$, independent of anything important, such that $A\leq CB$.  Notionally, if we say that a random variable has some size \(f(n)\), then all \(L^p\)-norms are bounded by a constant depending on $p$ times \(f(n)\).

In the sequel, we will make repeated use of the following $L^p$-estimates of the increments of $X^H$.
\begin{lemma}\label{lem:fBM-bound}
Recall that \(X^H_t = K_H^{-1}\int_0^t(t-s)^{H -1/2}\sigma_sdB_s\).  For any $p>0$,
 \begin{equation}\label{eq:fBM-bound}
\lpnorm{\Delta^n_i X^H}{p} \lec n^{-H}, 
\end{equation}
with a constant that does not depend on $n$ or $i$.
\end{lemma}
\begin{proof} 
For simplicity, we write \(X^H_t =   \int_0^t h(t-s)\sigma_sdB_s \) where $h(t)=K_H^{-1}t^{H-1/2}_+$. Then by the Burkholder--Davis--Gundy (BDG) inequality,
    \begin{equation*}
\lpnorm{\Delta^n_i X^H}{p}  \lec  \E\bigg[\bigg|\int^{\frac{i}{n}}_0 \big(h(\tfrac{i}{n}-s)-h(\tfrac{i-1}{n}-s)\big)^2\sigma_s^2 ds\bigg|^{\frac{p}{2}}\bigg]^{\frac{1}{p}}.
    \end{equation*}
If $p\geq2$, we apply the Minkowski integral inequality to bound the previous line by
    \begin{equation*}
 \bigg(\int_0^{\frac i n} \big(h(\tfrac{i}{n}-s)-h(\tfrac{i-1}{n}-s)\big)^2 \lpnorm{\sigma^2_s}{\frac{p}{2}} ds\bigg)^{\frac{1}{2}}\lec  \bigg(\int_0^{\frac i n} \big(h(\tfrac{i}{n}-s)-h(\tfrac{i-1}{n}-s)\big)^2   ds\bigg)^{\frac{1}{2}}
    \end{equation*}
    since \(\sigma\) has uniformly bounded moments. Thus,
    \begin{equation*} 
  \lpnorm{\Delta^n_i X^H}{p}\lec \bigg(\int_0^{\frac i n} \big(h(s)-h(s-\tfrac1n)\big)^2   ds\bigg)^{\frac{1}{2}}\leq n^{-H}\bigg(\int_0^\infty \big(h( s)-h( s-1)\big)^2   ds\bigg)^{\frac{1}{2}},
    \end{equation*}
which shows \eqref{eq:fBM-bound} for $p\geq2$ since the last integral is finite.
 if \(p < 2\), we simply apply  Jensen's inequality.
\end{proof}

Next, we  compute $L^p$-bounds on an averaged increment of $X^H$. The result  motivates the choice of normalization in  the definition of \(V(g)_T^{n,f}(Y)\).
\begin{prop}\label{prop:Lp}
  For every $p\geq2$, we have
  \begin{equation}\label{eq:Lp}
    \lpnorm{\overline{X^H}(g)^n_i}{p} \lec (  {k_n}/{n}  )^H,
  \end{equation}
with a constant that does not depend on $n$ or $i$.
\end{prop}

\begin{proof}
  Recall the notation
\eqref{eq:prev-W} for a general process $W$ as well as $
    \Delta g^n_j = g^n_j-g^n_{j-1}
$,
 with $ g^n_{-1} = 0$. To get an idea of the proof, we consider
 the special case when \(p=2\) first. We can apply summation by parts and obtain
  \begin{equation*}\begin{split}
    \E\Big[ \Big(\overline{X^H}(g)^n_i\Big)^2 \Big]
    &
      = \E\bigg[ \bigg( \sum_{j=1}^{k_n-1} g^n_j \Delta_{i+j-1}^nX^H \bigg)^2 \bigg]  =\E\bigg[ \bigg( \sum_{j=1}^{k_n-1} \sum_{l=0}^{j} 
                         \Delta g^n_l\Delta_{i+j-1}^nX^H \bigg)^2  \bigg]\\
                         &=\E\bigg[ \bigg( \sum_{l=0}^{k_n-1} \Delta g^n_l \sum_{j=l}^{k_n-1} \Delta_{i+j-1}^nX^H 
\bigg)^2  \bigg].
    \end{split}
\end{equation*}
By   the Cauchy--Schwarz inequality,
\begin{align*} 
    \E\bigg[ \bigg( \sum_{l=0}^{k_n-1} \Delta g^n_l \Big(X^H_{\frac{i+k_n-2}{n}} - X^H_{\frac{i+l-2}{n}}\Big)\bigg)^2  \bigg] &\leq 2\bigg(\sum_{l=0}^{k_n-1}1\bigg) \bigg( \sum_{l=0}^{k_n-1} \E\bigg[  ( \Delta g^n_l  )^2 \Big(X^H_{\frac{i+k_n-2}{n}} - X^H_{\frac{i+l-2}{n}}\Big)^2 \bigg]\bigg)\\
      &= 2 k_n \sum_{l=0}^{k_n-1}  ( \Delta g^n_l  )^2 \E\bigg[  \Big(X^H_{\frac{i+k_n-2}{n}} - X^H_{\frac{i+l-2}{n}}\Big)^2 \bigg]\\
& \lec k_n \sum_{l=0}^{k_n-1} \frac{1}{k_n^2}  (g'(\xi^n_l) )^2  \bigg(\frac{k_n-l}{n}\bigg)^{2H}, 
    \end{align*}
where the last inequality follows from the mean-value theorem and  Lemma~\ref{lem:fBM-bound}. 
Recall that \(g'\) is bounded, so this is bounded as
    \begin{equation*}
\E\bigg[ \bigg( \sum_{l=0}^{k_n-1} \Delta g^n_l \Big(X^H_{\frac{i+k_n-2}{n}} - X^H_{\frac{i+l-2}{n}}\Big)\bigg)^2  \bigg] \lec \frac{1}{k_n}\sum_{l=0}^{k_n-1}\left(\frac{k_n-l}{n}\right)^{2H} \leq  ( {k_n}/{n} )^{2H}, 
   \end{equation*}
which is \eqref{eq:Lp} for $p=2$.

  Then we can extend this to $L^p$, with $p \geq2$ by using H\"older's inequality and Lemma~\ref{lem:fBM-bound} to obtain
  \begin{align*} 
 \E\Big[ \Big \lvert\overline{X^H}(g)^n_i\Big \rvert^p \Big] &\lec \E\bigg[ \bigg\lvert \sum_{l=0}^{k_n-1} \Delta g^n_l \left(X^H_{\frac{i+k_n-2}{n}} - X^H_{\frac{i+l-2}{n}}\right)\bigg\rvert^p  \bigg] \\
     & \leq k_n^{p-1}   \sum_{l=0}^{k_n-1} \lvert \Delta g^n_l \rvert^p \E\bigg[\Big|X^H_{\frac{i+k_n-2}{n}} - X^H_{\frac{i+l-2}{n}}\Big|^p \bigg] \\
     &\lec   k_n^{p-1} \sum_{l=0}^{k_n-1}\bigg\lvert\frac{g'(\xi^n_l)}{k_n}\bigg\rvert^p\left(\frac{k_n-l}{n}\right)^{pH}\\
& \lec  ( {k_n}/{n} )^{pH}.\qedhere
    \end{align*}
\end{proof}

 The next goal is to remove the drift $A$, that is, we show that we can safely work with only the fractional component $X^H_t$ of \(X_t\).
 
\begin{lemma}
Letting
  \begin{equation}
    U_t = Y_t - X_0 - A_0 = X_t^H + \rho_t Z_t,
  \end{equation}
we have that 
  \begin{equation*}
    \lim_{n \rightarrow \infty} \lpnorm{ V(g)^{n,f}_T(Y) - V(g)^{n,f}_T(U)}{p} = 0.
  \end{equation*}
\end{lemma}

\begin{proof}
  Recalling the definition of the variation, the difference is really
  \begin{equation*}
    \lpnorm{\frac{1}{n}\sum_{i=1}^{ \lfloor nT  \rfloor} \bigg[ f\bigg( \frac{\overline{Y}(g)^n_i}{ ( k_n/n  )^H}, \frac{\widehat{Y}(g)^n_i}{ ( k_n/n  )^{2H}} \bigg) - f\bigg( \frac{\overline{U}(g)^n_i}{ ( k_n/n  )^H}, \frac{\widehat{U}(g)^n_i}{ ( k_n/n  )^{2H}} \bigg) \bigg] }{p}.
  \end{equation*}
  Applying the mean-value theorem to the difference,  it becomes
  \begin{equation*}
    \frac{1}{n}\sum_{i=1}^{\left\lfloor nT \right\rfloor} \begin{pmatrix} \partial_1 f(\xi^n_{1,i}) \\ \partial_2 f(\xi^n_{2,i}) \end{pmatrix}^T \begin{pmatrix} \frac{\overline{Y}(g)^n_i - \overline{U}(g)^n_i}{(k_n/n)^H} \\ \frac{\widehat{Y}(g)^n_i - \widehat{U}(g)^n_i}{(k_n/n)^{2H}} \end{pmatrix}
  \end{equation*}
for some intermediate $\xi^n_{1,i}$ and $\xi^n_{2,i}$. Proposition~\ref{prop:Lp} and Lemma~\ref{lem:size-U} below show that $\xi^n_{1,i}$ and $\xi^n_{2,i}$ have uniformly bounded moments in $n$ and $i$ of all orders $p$. Since $\partial_1 f$ and $\partial_2 f$ are of polynomial growth, the same is true for $\partial_1 f(\xi^n_{1,i})$ and $\partial_2 f(\xi^n_{2,i})$. Thus, by an application of H\"older's inequality, it suffices to show that the two differences in the previous display converge to $0$ in $L^p$ for any $p>0$.

  Since \(\lVert{\Delta_{i+j-1}^n A}\rVert_{L^1} \lec n^{-1}\) by the boundedness of $a$, we have \(\lVert{\overline{A}(g)^n_i}\rVert_{L^1} \lec   {k_n}/{n}\). Thus,
  \begin{equation*}
    \lpnorm{\frac{\overline{Y}(g)^n_i - \overline{U}(g)^n_i}{(k_n/n)^H}}{p} = \lpnorm{\frac{\overline{A}(g)^n_i}{(k_n/n)^H} }{p} \lec  ( {k_n}/{n} )^{1 - H},
  \end{equation*}
  which tends to $0$ for any   \(\kappa \in(0, 1)\). For the other difference, we first bound
  \begin{align*}
    \lpnorm{\frac{\widehat{Y}(g)^n_i - \widehat{U}(g)^n_i}{(k_n/n)^{2H}}}{p} &=  \lpnorm{\frac{n^{2H}}{k_n^{2H}}\sum_{j=1}^{k_n}(\Delta g^n_i)^2\left( (\Delta^n_{i+j-1}Y)^2 - (\Delta^n_{i+j-1}U)^2 \right)}{p} \\
                                                                 &\leq \frac{n^{2H}}{k_n^{2H}}\sum_{j=1}^{k_n}(\Delta g^n_i)^2\left( \lpnorm{(\Delta^n_{i+j-1}A)^2}{p} + 2\lpnorm{\Delta^n_{i+j-1}A\Delta^n_{i+j-1}U}{p} \right).                                          
\end{align*}
Clearly, \(\lVert(\Delta_{i+j-1}^n A)^2\rVert_{L^1} \lec n^{-2}\). Since $Z$ has uniformly bounded moments of all orders and $\rho$ is bounded, we can use the H\"older's inequality  to get
\begin{align*}
  \lpnorm{\Delta^n_{i+j-1}A\Delta^n_{i+j-1}U}{p} &
  \leq \lpnorm{\Delta^n_{i+j-1}A}{2p}\left(\lpnorm{\Delta^n_{i+j-1}X^H}{2p}+\lpnorm{\Delta^n_{i+j-1}\rho Z}{2p}\right)\\
                                           &\lec n^{-1}\left(  ( {k_n}/{n} )^H + C \right) \lec n^{-1}.
\end{align*}
Using the standard bound $\lvert \Delta g^n_i\rvert \lec k_n^{-1}$, we conclude that 
\[
  \lpnorm{\frac{\widehat{Y}(g)^n_i - \widehat{U}(g)^n_i}{(k_n/n)^{2H}}}{p} \lec \frac{n^{2H}}{k_n^{2H}}\sum_{j=1}^{k_n} \left(\frac{g'\left( \xi^n_i \right)}{k_n}\right)^2n^{-1} \leq \frac{n^{2H-1}}{k_n^{2H+1}}\]
  which, since \(\kappa \geq \frac{2H}{2H + 1} > \frac{2H-1}{2H+1}\), converges to $0$ as \(n \rightarrow \infty\). 
\end{proof}

Note that the lower bound on \(\kappa\) in Theorem~\ref{thm:main} is not the most parsimonious choice possible just examining the previous proof;   the next step we take justifies a lower bound of \(\frac{2H}{2H+1}\) rather than \(\frac{2H-1}{2H+1}\) for $\kappa$. 

\begin{lemma}\label{lem:size-U}
We have
 \[ \limsup_{n \rightarrow \infty}\lpnorm{\frac{\widehat{U}(g)^n_i}{(k_n/n)^{2H}}}{p} \begin{cases} <\infty &\text{if } \kappa =\frac{2H}{2H+1},\\
=0&\text{if } \kappa> \frac{2H}{2H+1}.\end{cases}\]
\end{lemma}

\begin{proof}
Because \begin{equation}\label{eq:help} 
	\lpnorm{\frac{\widehat{X^H}(g)^n_i}{( {k_n}/{n})^{2H}}}{p}\leq \frac{n^{2H}}{k_n^{2H}} \sum_{j=1}^{k_n}|\Delta g^n_j|^2\lpnorm{(\Delta^n_{i+j-1}X^H)^2}{p} \lec \frac{1}{k_n^{2H+1}} \rightarrow 0,
\end{equation}  we   only need to worry about the noise part, which can be bounded by
  \begin{equation*}
    \lpnorm{\frac{\widehat{\rho Z}(g)^n_i}{(k_n/n)^{2H}}}{p} \leq \frac{n^{2H}}{k_n^{2H}}\sum_{j=1}^{k_n}\left( \Delta g^n_j \right)^2\lpnorm{(\Delta_{i+j-1}^n \rho Z)^2}{p} = \frac{n^{2H}}{k_n^{2H}}\sum_{j=1}^{k_n} \bigg(\frac{g' ( \xi^n_i  )}{k_n}\bigg)^2 \lpnorm{(\Delta_{i+j-1}^n \rho Z)^2}{p}.\end{equation*}
Now, since $\rho$ is bounded and $Z$ has uniformly bounded $p$th moments, the $L^p$-norm may vary with \(p\), but not \(n\), so the above is bounded by
\begin{equation}\label{eq:help2}
                                                           C_p\frac{n^{2H}}{k_n^{2H}}\sum_{j=1}^{k_n}\left( \frac{g'(\xi^n_i)}{k_n} \right)^2  \lec\frac{n^{2H}}{k_n^{2H+1}},
  \end{equation}
which implies the claim since $k_n \sim n^\kappa/\theta$.
\end{proof}

In comparison with semimartingales, 
the problematic part of working with fractional processes  is that the domain of integration in   stochastic integrals stretches all the way from \(0\). In the semimartingale case (i.e., with \(H = \frac{1}{2}\)), we have that an increment \(X^{1/2}_{t_2} - X^{1/2}_{t_1} = \int_{t_1}^{t_2} \si_s dB_s\), but for $H\neq \frac12$,
\[
  X^H_{t_2} - X^H_{t_1} = K_H^{-1}\int_0^{t_2} \Big[(t_2- s)^{H - \frac{1}{2}} - (t_1 - s)_+^{H - \frac{1}{2}} \Big]\si_sdB_s.
\]
What we would like to be able to do is to truncate this interval arbitrarily close to where we begin averaging, namely at \(\frac{i}{n}\) in the case of \(\overline{U}(g)^n_i\). To this purpose, we introduce \textit{truncated} increments of  \(X^H\), defined as
  \begin{equation}
    \Delta_{i}^{n,\epsilon} X^H= K_H^{-1} \int_{\frac{i}{n}-\epsilon}^{\frac{i}{n}} \Big[  ( \tfrac{i}{n} -s  )^{H - \frac{1}{2}} -  ( \tfrac{i-1}{n} -s  )_+^{H - \frac{1}{2}} \Big]\sigma_s dB_s,
  \end{equation}
where $\eps>0$ is a truncation parameter. Correspondingly, we define
  \begin{equation}
    \Delta_{i}^{n,\epsilon} U = \Delta_{i}^{n,\epsilon} X^H+ \Delta_{i}^n \rho Z,\quad \overline{U}(g)^{n,\eps}_i = \sum_{j=1}^{k_n-1}g^n_j\Delta_{i+j-1}^{n,\eps} U,  \quad
    \widehat{W}(g)^n_i =   \sum_{j=1}^{k_n} (\Delta g^n_j )^2 (\Delta_{i+j-1}^{n,\eps} U)^2.
  \end{equation}

In order for this truncation to be a reasonable endeavor, we need the error we are making to be small; the amount which we discard must vanish in the limit if we take \(\epsilon \rightarrow 0\).

\begin{lemma} For any $\eps>0$,
  \begin{equation}
    \limsup_{n \rightarrow \infty} \lpnorm{V(g)^{n,f}_T(U) - \frac{1}{n}\sum_{i= \lfloor n\epsilon  \rfloor + 1}^{ \lfloor nT  \rfloor} f
\bigg( \frac{\overline{U}(g)^{n,\epsilon}_i}{(k_n/n)^H}, \frac{\widehat{U}(g)^{n,\epsilon}_i}{(k_n/n)^{2H}} \bigg)}{p} \lec \epsilon.
  \end{equation}
\end{lemma}

\begin{proof}
  The difference is composed of two parts:
  \begin{gather}
    \label{eqn:truncdif1}
    \frac{1}{n}\sum_{i=1}^{ \lfloor n\epsilon  \rfloor} f\bigg( \frac{\overline{U}(g)^n_i}{(k_n/n)^H}, \frac{\widehat{U}(g)^n_i}{(k_n/n)^H} \bigg), \\
    \label{eqn:truncdif2}
    \frac{1}{n}\sum_{i= \lfloor n\epsilon  \rfloor + 1}^{ \lfloor nT  \rfloor}\bigg[ f \bigg( \frac{\overline{U}(g)^{n}_i}{(k_n/n)^H}, \frac{\widehat{U}(g)^{n}_i}{(k_n/n)^H} \bigg) -  f \bigg( \frac{\overline{U}(g)^{n,\epsilon}_i}{(k_n/n)^H}, \frac{\widehat{U}(g)^{n,\epsilon}_i}{(k_n/n)^H} \bigg) \bigg ].
  \end{gather}
The \(f(\cdot)\)-term in \eqref{eqn:truncdif1} is of size \(1\), since \(f\) is assumed to be of polynomial growth and we showed previously that the arguments are at worst of size \(1\). This show that \eqref{eqn:truncdif1} is $\lec \eps$. Next, we want to show that   \eqref{eqn:truncdif2} converges to $0$. 

The mean-value theorem reduces (\ref{eqn:truncdif2}) to
  \begin{gather}
    \frac{1}{n}\sum_{ \lfloor n\epsilon  \rfloor + 1}^{ \lfloor nT  \rfloor} \begin{pmatrix} \partial_1 f(\xi_{1,i}^n) \\ \partial_2 f(\xi_{2,i}^n) \end{pmatrix}^T   \begin{pmatrix} \frac{\overline{U}(g)^{n,\epsilon}_i - \overline{U}(g)^{n}_i}{(k_n/n)^H} \\ \frac{\widehat{U}(g)^{n,\epsilon}_i - \widehat{U}(g)^{n}_i}{(k_n/n)^{2H}} \end{pmatrix}.
  \end{gather}
As seen before, the partial derivatives are of size $1$, so all that matters is that the differences on the right go to $0$.
  Explicitly,
  \begin{align*}
    \overline{U}(g)^{n}_i - \overline{U}(g)^{n,\epsilon}_i &= \sum_{j=1}^{k_n-1}g^n_i(\Delta_{i+j-1}^{n}U - \Delta_{i+j-1}^{n,\epsilon}U) \\
                                                           &= K_H^{-1} \sum_{j=1}^{k_n-1}g^n_i\int_0^{\frac{i+j-1}{n}-\epsilon}\Big[  ( \tfrac{i+j-2}{n} -s  )^{H -  {1}/{2}} -  ( \tfrac{i+j-1}{n} -s  )^{H -  {1}/{2}} \Big]\sigma_sdB_s,
\end{align*}
so 
   by the BDG inequality, the boundedness of $\si$ and a substitution \(r = \frac{i+j-1}{n}-s\),
\begin{align*}
    \lpnorm{\overline{U}(g)^{n}_i - \overline{U}(g)^{n,\epsilon}_i}{p} &\lec \bigg(\int_0^{\frac{i+j-1}{n}-\epsilon} \Big[  ( \tfrac{i+j-1}{n} -s  )^{H -  {1}/{2}} -  ( \tfrac{i+j-2}{n} -s  )^{H -  {1}/{2}} \Big]^2 ds\bigg)^{\frac{1}{2}}
                                                             \\
&=   \bigg( \int_{\epsilon }^{\frac{i+j-1}{n}} \Big( r^{H- {1}/{2}} - (r-\tfrac1n)^{H-\frac{1}{2}} \Big)^2dr\bigg) ^{\frac{1}{2}} 
\end{align*}
By the mean-value theorem, it follows that 
 \begin{align*}  \lpnorm{\overline{U}(g)^{n}_i - \overline{U}(g)^{n,\epsilon}_i}{p} &\lec  n^{-1} \bigg( \int_\epsilon^{\infty}   r^{2H-3}  dr\bigg) ^{\frac{1}{2}} \lec C(\eps)n^{-1}.
  \end{align*}
  Normalizing by \((k_{n}/n)^{H}\), we get that \[\frac{\lpnorm{\overline{U}(g)^{n}_i - \overline{U}(g)^{n,\epsilon}_i}{p}}{(k_{n}/n)^{H}} \lec C(\eps)n^{-(1-H)-\kappa H}\to0.\] 

  Next,
  \begin{align*}
    \widehat{U}(g)^{n}_i - \widehat{U}(g)^{n,\epsilon}_i &= \sum_{j=1}^{k_n-1}(\Delta g^n_i)^2\Big((\Delta_{i+j-1}^{n}U)^2 - (\Delta_{i+j-1}^{n,\epsilon}U)^2\Big) \\
                                                         &= \sum_{j=1}^{k_n-1}(\Delta g^n_i)^2(\Delta_{i+j-1}^{n}U - \Delta_{i+j-1}^{n,\epsilon}U)(\Delta_{i+j-1}^{n}U + \Delta_{i+j-1}^{n,\epsilon}U)
\end{align*}
We just saw that $\Delta_{i+j-1}^{n}U - \Delta_{i+j-1}^{n,\epsilon}U$ is of size $n^{-1}$, while $\Delta_{i+j-1}^{n}U + \Delta_{i+j-1}^{n,\epsilon}U$ is of size $1$. Thus, $\widehat{U}(g)^{n}_i - \widehat{U}(g)^{n,\epsilon}_i $ is of size $k_n^{-1}n^{-1}$, which still  goes to $0$ after dividing by $(k_n/n)^{2H}$.
\end{proof}

Our next step is to discretize \(\sigma\) and remove the fractional portion from the second argument.

\begin{lemma} We have
  \begin{multline*}
      \lim_{\epsilon \rightarrow 0}\limsup_{n \rightarrow \infty} \Bigg\lVert \frac{1}{n}\sum_{i= \lfloor n\epsilon  \rfloor + 1}^{ \lfloor nT  \rfloor} \bigg[ f\bigg( \frac{\overline{U}(g)^{n,\epsilon}_i}{(k_n/n)^H}, \frac{\widehat{U}(g)^{n,\epsilon}_i}{(k_n/n)^{2H}} \bigg)\\ 
      -f\bigg( \frac{\sigma_{\frac{i}{n}-\epsilon}\overline{B^{H}}(g)^{n,\epsilon}_i +  \overline{
      		\rho Z}(g)^n_i}{(k_n/n)^{H}}, \frac{\widehat{\rho Z}(g)^n_i}{(k_n/n)^{2H}}\bigg) \bigg] \Bigg\rVert_{L_p} = 0
  \end{multline*}
  where \(B^{H}_t=K_H^{-1}\int_0^t (t-s)^{H-1/2}dB_s\) is fractional Brownian motion with unit volatility.
\end{lemma}

\begin{proof}
  Via the mean-value theorem, the difference becomes
  \begin{equation}
    \frac{1}{n}\sum_{i= \lfloor n\epsilon  \rfloor + 1}^{ \lfloor nT  \rfloor} \begin{pmatrix} \partial_1 f(\xi_{1,i}^n) \\ \partial_2 f(\xi_{2,i}^n) \end{pmatrix}^T \begin{pmatrix} \frac{\overline{X^H}(g)^{n,\epsilon}_i - \sigma_{ {i}/{n}-\epsilon}\overline{B^{  H}}(g)^{n,\epsilon}_i}{(k_n/n)^H} \\ \frac{\widehat{U}(g)^{n,\epsilon}_i - \widehat{\rho Z}(g)^{n}_i}{(k_n/n)^{2H}} \end{pmatrix},
  \end{equation}
where the exact values of $\xi^n_{1,i}$ and $\xi^{n}_{2,i}$ may be different from before.
Again, all we care about are the differences
  \begin{equation}
    \label{eqn:discdiff1}
    \frac{\overline{X^{H}}(g)^{n,\epsilon}_i - \sigma_{\frac{i}{n}-\epsilon}\overline{B^{\prime H}}(g)^{n,\epsilon}_i}{(k_n/n)^H},\qquad
    \frac{\widehat{U}(g)^{n,\epsilon}_i - \widehat{\rho Z}(g)^{n}_i}{(k_n/n)^{2H}}.
  \end{equation}
In order to handle the first term in
    (\ref{eqn:discdiff1}), we recall $h$ from the proof of Lemma~\ref{lem:fBM-bound} and write
  \begin{align*}
    \overline{X^{H}}(g)_{i}^{n,\epsilon} - \sigma_{\frac{i}{n}-\epsilon}\overline{B^{  H}}(g)^{n,\epsilon}_{i} &= \sum_{j=1}^{k_n-1}g_j^n\Delta^{n, \epsilon}_{i+j-1}X^H - \sum_{j=1}^{k_n-1}g_j^n\sigma_{\frac{i}{n}-\epsilon}\Delta^{n, \epsilon}_{i+j-1}B^{ H}\\
    &= \sum_{j=1}^{k_n-1}g_j^n  \int^{\frac{i+j-1}{n}}_{\frac{i+j-1}{n}-\epsilon}\Big(h(\tfrac{i+j-1}{n}-s)-h(\tfrac{i+j-2}{n}-s)\Big)(\sigma_s-\si_{\frac in-\eps})dB_s.
\end{align*}
Recall that the lower bound in the above integral is a consequence of the truncation from the previous lemma, where we removed the integral from $0$ to $\frac{i+j-1}{n}-\eps$. Clearly, we could have only removed the integral from $0$ to $\frac{i}{n}-\eps$, leaving the portion from $\frac{i}{n}-\eps$ to $\frac{i+j-1}{n}$. In other words, there is no harm to replace the integral in the last display by
\[\Delta^{n,\epsilon}_{i,j}\wt{X^H}=\int^{\frac{i+j-1}{n}}_{\frac{i}{n}-\epsilon}\Big(h(\tfrac{i+j-1}{n}-s)-h(\tfrac{i+j-2}{n}-s)\Big)(\sigma_s-\si_{\frac in-\eps})dB_s\]  and \(\wt{X}^{H,n,\epsilon}_{i}= X^{H n,\epsilon}_{i} - \sigma_{\frac{i}{n}-\epsilon}B^{\prime H n, \epsilon}_{i}\).
    Then, similarly to the proof of Proposition~\ref{prop:Lp}, 
    \begin{equation*}
      \lpnorm{\sum_{j=1}^{k_n-1}g_j^n\Delta^{n,\epsilon}_{i,j}\wt{X^H}}{p} = \lpnorm{\sum_{j=1}^{k_n-1}\sum_{l=0}^{j}\Delta g^n_l\Delta^{n,\epsilon}_{i,j}\wt{X^H}}{p} \leq   \lpnorm{\sum_{l=0}^{k_n-1}\Delta g^n_l\sum_{j=l}^{k_n-1}\Delta^{n,\epsilon}_{i,j}\wt{X^H}}{p}.
    \end{equation*}
By definition,
\[ \sum_{j=l}^{k_n-1}\Delta^{n,\epsilon}_{i,j}\wt{X^H} =\int^{\frac{i+k_n-2}{n}}_{\frac{i}{n}-\epsilon}\Big(h(\tfrac{i+k_n-2}{n}-s)-h(\tfrac{i+l-2}{n}-s)\Big)(\sigma_s-\si_{\frac in-\eps})dB_s,  \] 
which is a truncated fractional increment over an interval of length $(k_n-\ell)/n$ with volatility $\si_s-\si_{i/n-\eps}$. By Lemma~\ref{lem:fBM-bound} and the $L^p$-continuity of $\si$, we can therefore bound
\[ \lpnorm{ \sum_{j=l}^{k_n-1}\Delta^{n,\epsilon}_{i,j}\wt{X^H}}{p} \lec (k_n/n)^{H}\phi_p(2\eps), \]
where $\phi_p(\tau)= \sup_{s,t\in[0,T],\lvert s-t\rvert \leq \tau}\E[\lvert \si_t-\si_s\rvert^p]^{1/p}$ satisfies $\phi_p(\tau)\to0$ as $\tau\to0$.
Applying H\"older's inequality, we then get
    \begin{equation*}
   \lpnorm{\sum_{l=0}^{k_n-1}\Delta g^n_l\sum_{j=l}^{k_n-1}\Delta^{n,\epsilon}_{i,j}\wt{X^H}}{p} \leq k_n^{1-1/p}\Bigg(\sum_{l=0}^{k_n-1}\lvert\Delta g^n_l\rvert^p\E\bigg[\bigg|\sum_{j=l}^{k_n-1}\Delta^{n,\epsilon}_{i,j}\wt{X^H} \bigg|^p\bigg]\Bigg) ^{\frac{1}{p}}
    \lec  ( {k_n}/{n} )^{H}\phi_p(2\eps).
    \end{equation*}
Dividing by $(k_n/n)^H$, we can first let $n\to\infty$ and then $\eps\to0$ to show that the first term in \eqref{eqn:discdiff1} is negligible.

Regarding the second term in \eqref{eqn:discdiff1}, we decompose
\begin{equation*}
\widehat{U}(g)^{n,\epsilon}_i - \widehat{\rho Z}(g)^{n}_i	=\sum_{j=1}^{k_n} (\Delta g^n_j)^2 (\Delta^n_{i+j-1} X^H)^2 +\sum_{j=1}^{k_n} (\Delta g^n_j)^2 (\Delta^n_{i+j-1} X^H)(\Delta^n_{i+j-1} \rho Z).
\end{equation*}
We have already shown in \eqref{eq:help} that the first term divided by $(k_n/n)^{2H}$ converges to $0$ in $L^p$. Applying the Cauchy--Schwarz inequality, we can bound the second term by
\[\Bigg(\sum_{j=1}^{k_n} (\Delta g^n_j)^2 (\Delta^n_{i+j-1} X^H)^2\Bigg)^{1/2}\Bigg(\sum_{j=1}^{k_n} (\Delta g^n_j)^2 (\Delta^n_{i+j-1} \rho Z)^2\Bigg)^{1/2},\]
so \eqref{eq:help} and \eqref{eq:help2} imply that the second term is also negligible.
\end{proof}

The advantage of having both truncated the fractional increments and discretized $\si$ is that we can now prove a law of large numbers.
\begin{lemma} We have
  \begin{multline*}
  \lim_{\eps\to0}\limsup_{n\to\infty}  \Bigg\lVert \frac{1}{n}\sum_{i =  \lfloor n\epsilon  \rfloor + 1}^{ \lfloor nT  \rfloor} \Bigg\{   f\bigg( \frac{\sigma_{\frac{i}{n}-\epsilon}\overline{B^H}(g)^{n,\epsilon}_i + \overline{\rho Z}(g)^n_i}{(k_n/n)^{2H}}, \frac{\widehat{\rho Z}(g)^n_i}{(k_n/n)^{2H}}\bigg)  \\    -  \E\bigg[ f\bigg( \frac{\sigma_{\frac{i}{n}-\epsilon}\overline{B^{ H}}(g)^{n,\epsilon}_i + \overline{\rho Z}(g)^n_i}{(k_n/n)^{H}}, \frac{\widehat{\rho Z}(g)^n_i}{(k_n/n)^{2H}}\bigg) \mathrel{\Big|} \mathcal{F}_{\frac{i}{n}-\epsilon} \bigg] \vphantom{\frac{1}{n}\sum_{i = \left\lfloor n\epsilon \right\rfloor + 1}^{\left\lfloor nT \right\rfloor}} \Bigg\} \Bigg\rVert=0.
  \end{multline*}
\end{lemma}

\begin{proof}
  Abbreviate \(f(\cdot)\) to \(f_{i}\). We have to show that
  \begin{equation}
    \lim_{\eps\to0}\limsup_{n\to\infty}  \E \bigg[ \bigg( \sum_{i =\lfloor n\eps\rfloor +1}^{\lfloor nT \rfloor} \Big\{ f_i - \E  [ f_i \mid \mathcal{F}_{\frac{i}{n}-\epsilon}  ]\Big\} \bigg)^2 \bigg] = 0.
  \end{equation}
 Developing the square, this becomes
  \begin{equation}
    \label{eqn:centering1}
    \frac{1}{n^2} \sum_{i, j = \lfloor n\epsilon\rfloor + 1}^{\lfloor nT \rfloor} \E \Big[ \Big\{ f_i - \E[ f_i \mid \mathcal{F}_{\frac{i }{n}-\epsilon} ] \Big\} \Big\{f_j - \E[ f_j \mid \mathcal{F}_{\frac{j}{n}-\epsilon} ] \Big\} \Big]
  \end{equation}
  By conditioning, we can see that as soon as \( ( \frac{i}{n} - \epsilon, \frac{i + k_{n}}{n}  )\) and \( ( \frac{j}{n} - \epsilon, \frac{j + k_{n}}{n}  )\) become disjoint, then this covariance is actually $0$. 
  So in (\ref{eqn:centering1}), at most $\lfloor nT\rfloor\times (2k_n+2n\eps)$ terms are non-zero. Bounding each non-zero term by $\E[f_i^2]\leq C$, we can bound \eqref{eqn:centering1} by
  \begin{equation}
    \frac{1}{n^2} \times \lfloor nT \rfloor \times (2 k_n + 2n\epsilon) \times C \sim C\epsilon \rightarrow 0
  \end{equation}
  as \(\epsilon \rightarrow 0\). 
\end{proof}

The next lemma calculates the conditional expectations arising in the law of large numbers.

\begin{lemma} Recalling the function $\mu_f$ from \eqref{eq:muf}, we have
  \begin{multline*}
\lim_{\eps\to0}\limsup_{n \rightarrow \infty} \sum_{i=\lfloor n\epsilon \rfloor + 1}^{\lfloor nT \rfloor} \Bigg\{ \E\bigg[f\bigg( \frac{\sigma_{ {i}/{n}-\epsilon}\overline{B^{  H}}(g)^{n,\epsilon}_i + \overline{\rho Z}(g)^n_i}{(k_n/n)^{H}}, \frac{\widehat{\rho Z}(g)^n_i}{(k_n/n)^{2H}}\bigg) \mathrel{\Big|} \mathcal{F}_{\frac{i-1}{n}-\epsilon} \bigg] \\
 \mu_f\bigg( \frac{\sigma^2_{ {i}/{n}-\epsilon}}{k_n^{2H}}\sum_{j,l = 1}^{k_n-1} g^n_jg^n_l\Gamma^H_{|j-l|}, \frac{\rho^2_{i/n-\eps}}{(k_n/n)^{2H}}\sum_{j= 1}^{k_n}(\Delta g^n_j)^2 \bigg) \Bigg\} = 0,
  \end{multline*}
where 
\[ \Ga^H_r=  \frac12((r+1)^{2H}-2r^{2H}+\lvert r-1\rvert^{2H}),\qquad r\geq0, \]
is the autocorrelation function of the increments of a standard fractional Brownian motion.
\end{lemma}

\begin{proof}
We first pass to the limit in the second argument, by showing  
  \begin{multline*}
 \lim_{\eps\to0} \limsup_{n \rightarrow \infty}  \sum_{i=\lfloor n\epsilon \rfloor + 1}^{\lfloor nT \rfloor} \E\bigg[\bigg(f\bigg( \frac{\sigma_{{i}/{n}-\epsilon}\overline{B^H}(g)^{n,\epsilon}_i + \overline{\rho Z}(g)^n_i}{(k_n/n)^{H}}, \frac{\widehat{\rho Z}(g)^n_i}{(k_n/n)^{2H}}\bigg)   \\
     - f\bigg( \frac{\sigma_{ {i}/{n}-\epsilon}\overline{B^H}(g)^{n,\epsilon}_i + \overline{\rho Z}(g)^n_i}{(k_n/n)^{H}}, \frac{2\rho^{2}_{i/n-\eps}}{(k_{n}/n)^{2H}}\sum_{j=1}^{k_{n}}(\Delta g^{n}_{j})^{2}\bigg)\bigg)\mathrel{\Big| }\mathcal{F}_{\frac{i}{n}-\epsilon} \bigg] = 0.
  \end{multline*}
  By another usage of the mean-value theorem as before, it is sufficient to show that as \(n \rightarrow \infty\),
  \begin{align}
 \lim_{n\to\infty}   \E \bigg[ \bigg| \frac{\widehat{\rho Z}(g)^{n}_{i}}{(k_{n}/n)^{2H}} - \frac{1}{(k_{n}/n)^{2H}}\sum_{j=1}^{k_{n}}(\Delta g^{n}_{j})^{2} (\rho^{2}_{  (i+j-1)/n}+\rho^2_{(i+j-2)/n})\bigg|^2 \bigg] = 0,\label{eq:toshow1}\\
   \lim_{\eps\to0}  \limsup_{n\to\infty} \E \bigg[ \bigg| \frac{1}{(k_{n}/n)^{2H}}\sum_{j=1}^{k_{n}}(\Delta g^{n}_{j})^{2} (\rho^{2}_{  (i+j-1)/n}+\rho^2_{(i+j-2)/n}) - \frac{2\rho^{2}_{  i/n-\eps}}{(k_{n}/n)^{2H}}\sum_{j=1}^{k_{n}}(\Delta g^{n}_{j})^{2} \bigg|^2 \bigg] = 0.\label{eq:toshow2}
  \end{align}
The second convergence  is easy, as the expectation in \eqref{eq:toshow2} is of order $((n^{2H}/k_n^{2H+1})\psi(k_n/n)+\eps)^2$, which goes to $0$ as $n\to\infty$ and $\eps\to0$ since $\psi(\tau)=\sup_{s,t\in[0,T],\lvert t-s\rvert \leq \tau} \E[\lvert \rho^2_t-\rho^2_s\rvert] \to0$ as $\tau\to0$ by the $L^2$-continuity of $\rho$. For \eqref{eq:toshow1}, we 
develop the square and obtain
  \begin{align*}
    & \frac{n^{4H}}{k_{n}^{4H}}\sum_{j,l = 1}^{k_{n}} (\Delta g^{n}_{j})^{2}(\Delta g^{n}_{l})^{2} \mathrm{Cov}\big((\Delta_{i+j-1}^{n}\rho Z)^{2},(\Delta_{i+l-1}^{n}Z)^{2} \big) 
  \end{align*}
 By the independence properties of $Z_t$, the covariance in the sum is $0$ unless \(|l - j| \leq 1\). Moreover, even if $|l-j|\leq 1$,  the covariance is bounded by a constant. Thus, the last display is $\lec n^{4H}k_{n}^{-4H} k_{n}k_{n}^{-4}$ and \(\Delta g^{n}_{i} \sim k_{n}^{-1}\), which tends to $0$ for any $\kappa\geq \frac{2H}{2H+1}$.

  To handle the first argument, observe that $ \sigma_{ {i}/{n}-\epsilon}\overline{B^H}(g)^{n,\epsilon}_i /{(k_n/n)^{H}}$
is $\calf_{i/n-\eps}$-conditionally Gaussian    with mean $0$ and variance
  \begin{multline*}
    \Var \bigg( \frac{\sigma_{ {i}/{n}-\epsilon}}{(k_{n}/n)^{H}}\sum_{j=\lfloor n\epsilon \rfloor + 1}^{k_{n}-1} g^{n}_{j}\Delta_{i+j-1}^{n,\epsilon}B^{  H}\mathrel{\Big|}\calf_{\frac in -\eps}\bigg) \\ =  \frac{\sigma^2_{ {i}/{n}-\epsilon}}{(k_{n}/n)^{2H}}\sum_{j,l = \lfloor n\epsilon \rfloor + 1}^{k_{n}-1}g^{n}_{j}g^{n}_{l}\Cov \Big( \Delta^{n,\epsilon}_{i+j-1}B^{  H}, \Delta^{n,\epsilon}_{i+l-1}B^{  H} \Big).
  \end{multline*}
In the same way we showed that we could truncate increments of $X^H$ or $B^H$ at level $\eps$, we can also reverse this step and approximate   the last display by
 \[    \frac{\sigma^2_{ {i}/{n}-\epsilon}}{(k_{n}/n)^{2H}}\sum_{j,l = \lfloor n\epsilon \rfloor + 1}^{k_{n}-1}g^{n}_{j}g^{n}_{l}\Cov \Big( \Delta^{n,\epsilon}_{i+j-1}B^{  H}, \Delta^{n,\epsilon}_{i+l-1}B^{  H} \Big)\approx \frac{\sigma^2_{{i}/{n}-\epsilon}}{k_n^{2H}}\sum_{j,l =\lfloor n\epsilon \rfloor + 1 }^{k_n-1} g^n_jg^n_l\Gamma^H_{|j-l|},\]
where the last step follows, for example, from \cite[Lemma~B.1]{CDL22}. In the last display $\approx_n$ means that the difference of the two sides goes to $0$ in $L^1$ as $n\to\infty$.

At the same time, 
\begin{equation}\label{eq:help3}
\frac{\ov{\rho Z}(g)^n_i}{(k_n/n)^H} = \frac{n^H}{k_n^H}\sum_{j=1}^{k_n-1} g^n_j \Delta^n_{i+j-1} \rho Z \approx_{n,\eps} \rho_{i/n-\eps}\frac{n^H}{k_n^H}\sum_{j=1}^{k_n-1} g^n_j \Delta^n_{i+j-1}  Z,
\end{equation}   
where $\approx_{n,\eps}$ means that  the difference of the two sides converges to $0$ in $L^1$ as $n\to\infty$ and $\eps\to0$. Since $g^n_j = g'(\xi^n_j)/k_n$,  we have by the  central limit theorem and the assumption $g(0)=g(1)=0$ that
\[ \sqrt{k_n}\sum_{j=1}^{k_n-1} g^n_j \Delta^n_{i+j-1}  Z= \sqrt{k_n}\bigg(g^n_{k_n-1}Z_{i+k_n-2}-g^n_1Z_{i-1} - \sum_{j=1}^{k_n-1} \Delta g^n_j  Z_{i+j-1}\bigg)  \stackrel{d}{\longrightarrow} N\bigg(0,\int_0^1 g(s)^2 ds\bigg). \]
If $\kappa>\frac{2H}{2H+1}$, it follows that \eqref{eq:help3} goes to $0$ in probability; if $\kappa= \frac{2H}{2H+1}$, \eqref{eq:help3} is approximately $\calf_{i/n-\eps}$-conditional  Gaussian with mean $0$ and variance $\theta^{2H+1}\rho_{i/n-\eps}^2 \int_0^1 g(s)^2 ds$. Since the last expression is also approximated by $\rho_{i/n-\eps}^2/(k_n/n)^{2H} \sum_{j=1}^{k_n} (\Delta g^n_j)^2$, the lemma is proved.
\end{proof}

Finally, we pass to the limit asserted in \eqref{eq:LLN}.
\begin{lemma} We have
  \begin{multline*}
\lim_{\epsilon \rightarrow 0} \limsup_{n\to\infty} \Bigg\lVert\frac{1}{n}\sum_{i =  \lfloor n\epsilon  \rfloor + 1}^{ \lfloor nT  \rfloor}\mu_f\bigg( \frac{\sigma^2_{ {i}/{n}-\epsilon}}{k_n^{2H}}\sum_{j,l = 1}^{k_n-1} g^n_jg^n_l\Gamma^H_{|j-l|}, \frac{\rho^2_{i/n-\eps}}{(k_n/n)^{2H}}\sum_{j=1}^{k_n}(\Delta g^n_j)^2 \bigg)  \\
            - \int_0^T\mu_f\bigg(\sigma_{s}^2\eta(g), \Theta^{2H+1}\rho^2_s\int_0^1g'(r)^2dr\bigg) ds\Bigg\rVert_{L_p}   = 0.
  \end{multline*}
\end{lemma}

\begin{proof} The difference above is $I^{n}_1+I^{n}_2$, where 
 \begin{align*}
I^{n}_1&=  \sum_{i =  \lfloor n\epsilon  \rfloor + 1}^{ \lfloor nT  \rfloor}\int_{\frac{i-1}{n}}^{\frac{i}{n}}\Bigg\{\mu_f\bigg( \frac{\sigma^2_{ {i}/{n}-\epsilon}}{k_n^{2H}}\sum_{j,l = 1}^{k_n-1} g^n_jg^n_l\Gamma^H_{|j-l|}, \frac{\rho^2_{i/n-\eps}}{( {k_n}/{n})^{2H}}\sum_{j=1}^{k_n}(\Delta g^n_j)^2 \bigg)\\
	&\qquad\qquad\qquad\qquad\qquad\qquad\qquad\qquad\qquad - \mu_f\bigg(\sigma_{s}^2\eta(g), \Theta^{2H+1}\rho^2_s\int_0^1g'(r)^2dr\bigg)ds\Bigg\}, \\
            I^n_2      &= -\int_0^{\frac{ \lfloor n\epsilon  \rfloor}{n}}\mu_f\bigg(\sigma_{s}^2\eta(g), \Theta^{2H+1}\rho^2_s\int_0^1g'(r)^2dr\bigg)ds- \int_{\frac{ \lfloor nT  \rfloor}{n}}^T\mu_f\bigg(\sigma_{s}^2\eta(g), \Theta^{2H+1}\rho_s^2\int_0^1g'(r)^2dr \bigg) ds.
 \end{align*}
 Since the \(\mu_f\)-term is of size one, $I^n_2$ is of size \(\epsilon+\frac1n\) and thus asymptotically negligible.
 
 Hence, we only need to show that $I^n_1$ goes to $0$. 
 By mean-value theorem, 
 \begin{multline*}
I^n_1= \sum_{i = \left\lfloor n\epsilon \right\rfloor + 1}^{ \lfloor nT  \rfloor}\int_{\frac{i-1}{n}}^{\frac{i}{n}}(\partial_1\mu_f (\xi^n_{1,i}), \partial_2\mu_f(\xi_{2,i}^n) )   \\
           \times\bigg(\frac{\sigma^2_{ {i}/{n}-\epsilon}}{k_n^{2H}}\sum_{j,l = 1}^{k_n-1} g^n_jg^n_l\Gamma^H_{|j-l|} - \sigma_{s}^2\eta(g),  \frac{\rho^2_{i/n-\eps}}{({k_n}/{n})^{2H}}\sum_{j=1}^{k_n}(\Delta g^n_j)^2 - \Theta^{2H+1}\rho^2_s\int_0^1 g'(r)^2dr \bigg)^T ds 
  \end{multline*}
  Each partial derivative is of size one, so we only need to show that the differences  vanish in $L^p$. To this end, because $\si^2_{i/n-\eps} \to \si_s^2$ and $\rho^2_{i/n-\eps}\to \rho^2_s$ in $L^p$ as $n\to\infty$ and $\eps\to0$, uniformly on $[0,T]$, it remains to show that
  \[\frac{1}{k_n^{2H}}\sum_{j,l = 1}^{k_n-1} g^n_jg^n_l\Gamma^H_{|j-l|} \to\eta(g),\qquad \frac{1}{({k_n}/{n})^{2H}}\sum_{j=1}^{k_n}(\Delta g^n_j)^2 \to\Theta^{2H+1}\int_0^1 g'(r)^2dr.  \]
  The second convergence was already shown in \eqref{eq:help3} and the subsequent display. The first one is of independent interest and will be shown in a separate lemma.
  \end{proof}

\begin{lemma}
Recall $\eta(g)$ from \eqref{eq:etag}.	As $n\to\infty$,
	\[\frac{1}{k_n^{2H}}\sum_{j,l = 1}^{k_n-1} g^n_jg^n_l\Gamma^H_{|j-l|} \to\eta(g).\]
\end{lemma}
  \begin{proof}
  Let us write \(\Gamma^H_{r} = \overline{\Gamma}^H_{r}- \overline{\Gamma}^H_{r-1} = \Delta \overline{\Gamma}^H_{r} \) with $\ov \Gamma^H_1=\frac12$.  
Using summation by parts, we have
 \begin{align*} 
   &\frac{1}{k_n^{2H}}\sum_{j = 1}^{k_n-1} g^n_jg^n_l\Gamma^H_{|j-l|} 
        = \frac{1}{k_n^{2H}}\sum_{j = 1}^{k_n-1} (g^n_j)^2 + \frac{2}{k_n^{2H}}\sum_{j= 1}^{k_n-2} g^n_j\bigg(\sum_{l= j+1}^{k_n-1} g^n_l\Delta\overline{\Gamma}^H_{l-j}\bigg)\\
        &\quad= \frac{1}{k_n^{2H}}\sum_{j = 1}^{k_n-1} (g^n_j)^2+ \frac{2}{k_n^{2H}}\sum_{j= 1}^{k_n-2} \bigg(g^n_jg^n_{k_n-1}\overline{\Gamma}^H_{k_n-1-j}- \frac{1}{2}g^n_jg^n_{j+1} - \sum_{l=j+1}^{k_n-2}g^n_j \Delta g^n_{l+1}\overline{\Gamma}^H_{l-j}\bigg)\\
        &\quad= \frac{1}{k_n^{2H}}\sum_{j = 1}^{k_n-1} g^n_j (g^n_j-g^n_{j+1} ) +\frac{2}{k_n^{2H}}\sum_{j = 1}^{k_n-2} g^n_j g^n_{k_n-1}\overline{\Gamma}^H_{k_n-1-j} - \frac{2}{k_n^{2H}}\sum_{j=1}^{k_n-3}\sum_{l=j+1}^{k_n-2} g^n_j \Delta g^n_{l+1}\overline{\Gamma}^H_{l-j}. 
    \end{align*}
By the mean-value theorem, the first term can be written as
 \begin{equation*}
-  \frac{1}{k_n^{2H}}\sum_{j = 1}^{k_n-1} g^n_j\Delta g^n_{j+1}
      = - \frac{1}{k_n^{2H}}\sum_{j = 1}^{k_n-1} g(\tfrac{j}{k_n}) (g(\tfrac{j+1}{k_n})-g(\tfrac{j}{k_n}) )= - \frac{1}{k_n^{2H+1}}\sum_{j = 1}^{k_n-1}g(\tfrac{j}{k_n})  {g'(\xi_j)}, 
    \end{equation*}
where \(\xi_j \in  (\frac{j}{k_n}, \frac{j+1}{k_n} )\). As \(n\rightarrow\infty\),  \(\frac1{k_n}\sum_{j = 1}^{k_n-1}g(\tfrac{j}{k_n})  {g'(\xi_j)}\to\int_0^1g(x)g'(x)dx\) by Riemann integration, so the term above vanishes as \(k_n\rightarrow \infty\).

Similarly, since $\ov \Ga^H_r = \frac12((r+1)^{2H}-r^{2H})$, we have
\begin{align*}
  \frac{2}{k_n^{2H}}\sum_{j = 1}^{k_n-2} g^n_j g^n_{k_n-1}\overline{\Gamma}^H_{k_n-1-j} &= \frac{1}{k_n^{2H}}\sum_{j = 1}^{k_n-2} g^n_j g^n_{k_n-1}  ((k_n-j)^{2H} - (k_n-j-1)^{2H} )\\
      &= \sum_{j = 1}^{k_n-2} g^n_j g^n_{k_n-1}\Big((1-\tfrac{j}{k_n})^{2H} - (1-\tfrac{j+1}{k_n})^{2H}\Big).
    \end{align*}
Let \(h(x) := (1-x)^{2H}\) and  \(h'(x) = -2H(1-x)^{2H-1}\). Then, again by Riemann integration,
\begin{align*}
    \frac{2}{k_n^{2H}}\sum_{j = 1}^{k_n-2} g^n_j g^n_{k_n-1}\overline{\Gamma}^H_{k_n-1-j} &= \sum_{j = 1}^{k_n-2} g^n_j g^n_{k_n-1}\Big(h(\tfrac{j}{k_n})^{2H} - h(\tfrac{j+1}{k_n})^{2H}\Big)= -\sum_{j = 1}^{k_n-2} g^n_j g^n_{k_n-1}\frac{h'(\xi^n_j)}{k_n}\\
  &  \to 2Hg(1)\int_0^1g(x)(1-x)^{2H-1}dx=0.
 \end{align*}
 Finally,
 \begin{align*}
  -\frac{2}{k_n^{2H}}\sum_{j=1}^{k_n-3}\sum_{l=j+1}^{k_n-2} g^n_j \Delta g^n_{l+1}\overline{\Gamma}^H_{l-j} = -\frac{1}{k_n^{2H}}\sum_{j=1}^{k_n-3} g^n_j \sum_{l=j+1}^{k_n-2}[(l-j+1)^{2H}-(l-j)^{2H}]\Delta g^n_{l+1}
    \end{align*}
    Thus, by summation by parts, this term is equal to
    \begin{align*}
    & -\frac{1}{k_n^{2H}}\sum_{j=1}^{k_n-3} g^n_j\bigg(\Delta g^n_{k_n-1}(k_n-1-j)^{2H} -\Delta g^n_{j+2} -\sum_{l=j+1}^{k_n-3} (l-j+1)^{2H}(\Delta g^n_{l+2}-\Delta g^n_{l+1})\bigg)\\
      &\quad= -\frac1{k_n}\sum_{j=1}^{k_n-3}\bigg[ g(\tfrac{j}{k_n}){g'(\xi^n_{k_n-1})}(1-\tfrac{j+1}{k_n})^{2H} -k_n^{-2H}g(\tfrac{j}{k_n}){g'(\xi^n_{j+2})} -\frac{g(\tfrac{j}{k_n})}{k_n}\sum_{l=j+1}^{k_n-3} (\tfrac{l-j+1}{k_n})^{2H} {g''(\xi^n_{l+2})}\bigg]\\
      &\quad\to 
      -g'(1)\int^1_0g(x)(1-x)^{2H}dx +\int_0^1\int_x^1g(x)(y-x)^{2H}g''(y)dydx 
      \end{align*}
Integrating by parts shows that the last line is equal to 
\[ -2H \int_0^1\int_x^1 g(x)(y-x)^{2H-1} g'(y) dy dx. \]
Interchanging the $dy$- with the $dx$-integral and integrating by parts one more time result in the form of $\eta(g)$ stated in \eqref{eq:etag}.
%
\end{proof}


\subsection*{Acknowledgments}
The Dome C data were acquired in the frame of the projects ``Mass lost in wind flux'' (MALOX) and ``Concordia multiprocess atmospheric studies'' (COMPASS) sponsored by PNRA. And a special thanks to Dr.\ Igor Petenko of CNR ISAC for running the field experiment at Concordia station.

	\bibliographystyle{abbrv}
\bibliography{mybib}

\end{document}